\documentclass[10pt,twoside]{article}

\usepackage{ecphead}

\usepackage{amsfonts}
\usepackage{graphicx}
\usepackage[latin1]{inputenc}
\usepackage{amsmath,amssymb,amsthm}
\usepackage{calc}
\usepackage{bbm}
\usepackage{cite}

\newtheorem{theorem}{Theorem}
\newtheorem{lemma}[theorem]{Lemma}
\newtheorem{corollary}[theorem]{Corollary}

\usepackage{bbm}
\usepackage{amsmath}
\usepackage{ifthen}
\usepackage{calc}
\usepackage{mathrsfs}

\usepackage{color}
\definecolor{light}{gray}{.91}

\newlength{\fboxsepdefault}
\newlength{\fboxsepnew}
\newlength{\SatzboxTextwidth}
\newlength{\GrauBoxTextwidth}
\newlength{\HauptsatzboxTextwidth}
\newlength{\hfillparboxwidth}
 \providecommand{\C}{{\ensuremath{\mathbf{C}}}}

 \providecommand{\E}{{\ensuremath{\mathbf{E}}}}

 \providecommand{\N}{{\ensuremath{\mathbbm{N}}}}

 \renewcommand{\P}{{\ensuremath{\mathbf{P}}}}
 
 \providecommand{\P}{{\ensuremath{\mathbf{P}}}}
 
 \providecommand{\R}{{\ensuremath{\mathbbm{R}}}}

 \providecommand{\MCG}{{\ensuremath{\mathcal G}}}

 \providecommand{\Sv} {{\ensuremath{\check{S}}}}

 \providecommand{\Zv} {{\ensuremath{\check{Z}}}}

 \providecommand{\Sh} {{\ensuremath{\hat{S}}}}

 \providecommand{\Zh} {{\ensuremath{\hat{Z}}}}

\providecommand{\rob}[1]   {{{\bigl(#1\bigr)}}}

\providecommand{\roB}[1]   {{{\Bigl(#1\Bigr)}}}

\providecommand{\sqb}[1]  {{{\bigl[#1\bigr]}}}

\providecommand{\delds}{{\ensuremath{\frac{\del}{\del s}}}}

\providecommand{\delds}{{\ensuremath{\frac{\del}{\del s}}}}
\providecommand{\delzdsz}{{\ensuremath{\frac{\del^2}{\del s^2}}}}

\providecommand{\delzdzds}{{\ensuremath{\frac{\del^2}{\del z\del s}}}}

\providecommand{\delds}{{\ensuremath{\frac{\del}{\del s}}}}
\providecommand{\delzdsz}{{\ensuremath{\frac{\del^2}{\del s^2}}}}

\providecommand{\deldz}{{\ensuremath{\frac{\del}{\del z}}}}
\providecommand{\delzdzz}{{\ensuremath{\frac{\del^2}{\del z^2}}}}

\providecommand{\floor}[1]  {{\ensuremath{\lfloor#1\rfloor}}}

\providecommand{\al}      {{\ensuremath{\alpha}}}

\providecommand{\ld}      {{\ensuremath{\lambda}}}

\providecommand{\eps}     {{\ensuremath{\varepsilon}}}

\providecommand{\limt}   {{\ensuremath{{\displaystyle \lim_{t \ra \infty}}}}}

\providecommand{\limK}   {{\ensuremath{{\displaystyle \lim_{K \ra \infty}}}}}
\providecommand{\limN}   {{\ensuremath{{\displaystyle \lim_{N \ra \infty}}}}}

\providecommand{\ra}{\rightarrow}

\providecommand{\lra}{\longrightarrow}

\providecommand{\wlimeps}%
      {{\ensuremath{\stackrel{\eps \rightarrow \infty}%
                            {\Longrightarrow}}}}

\providecommand{\varwlimn}  {\xrightarrow[n  \ra\infty]{\text{w}}}

\providecommand{\Gen}{{\ensuremath{\MCG}}}

\providecommand{\Law}[2][]{{\ensuremath{\mathcal L^{#1}\left(#2\right)}}}

\providecommand{\del}{{\ensuremath{\partial}}}

\providecommand{\clearemptydoublepage}%
    {\newpage{\pagestyle{empty}\cleardoublepage}}

\newlength{\mylen}

\def\mathclapinternal#1#2{\clap{$\mathsurround=0pt#1{#2}$}}
\def\clap#1{\hbox to 0pt{\hss#1\hss}}

\def\mathclap{\mathpalette\mathclapinternal}

\makeatletter
\newlength{\Breit@}\newlength{\breit@}\newlength{\diff@renz}
\providecommand{\St@ckrel}[2]{%
  \settowidth{\Breit@}{\ensuremath{^{#1}}}
  \settowidth{\breit@}{\ensuremath{#2}}
  \ifthenelse{\Breit@>\breit@}{
    \setlength{\diff@renz}{(\Breit@-\breit@)/2}
    \hspace{\diff@renz}
  }{}}

\providecommand{\etStackrel}[2]{%
  \St@ckrel{#1}{#2}%
  &\stackrel{\mathclap{#1}}{#2}
  \St@ckrel{#1}{#2}}
\makeatother

\renewcommand{\E}{\mathbb{E}}
\renewcommand{\P}{\mathbb{P}}

\begin{document}
%
%
%
%
%
%
%
\newcommand{\Title}{\uppercase{%
Supercritical branching diffusions\\ in random environment
}}
\newcommand{\ShortTitle}{%
Supercritical branching diffusions in random environment
}
%
%
%
\newcounter{NumAuthor}
\setcounter{NumAuthor}{%
1
}
\addtocounter{NumAuthor}{-1}
\newcommand{\AuthorOne}{\uppercase{%
Martin Hutzenthaler\footnote{Research supported by the Institute for Mathematical Sciences of the National University of Singapore}
}}
\newcommand{\AuthorTwo}{\uppercase{%
}}
\newcommand{\AuthorThree}{\uppercase{%
}}
\newcommand{\AuthorFour}{\uppercase{%
}}
\newcommand{\AuthorFive}{\uppercase{%
}}
%
%
%
\newcommand{\AddressOne}{%
Department of Biology,
Gro\ss haderner Str.~2,
82152 Planegg-Martinsried,
Germany
}
\newcommand{\AddressTwo}{%
}
\newcommand{\AddressThree}{%
}
\newcommand{\AddressFour}{%
}
\newcommand{\AddressFive}{%
}
%
%
%
\newcommand{\EmailOne}{%
hutzenthaler@bio.lmu.de
}
\newcommand{\EmailTwo}{%
email.com                               
}
\newcommand{\EmailThree}{%
email.com                               
}
\newcommand{\EmailFour}{%
email.com                               
}
\newcommand{\EmailFive}{%
email.com                               
}
%
%
\setcounter{volume}{8}                  
\setcounter{year}{2011}                     
\setcounter{firstpage}{781}               
\setcounter{lastpage}{791}                
\setcounter{page}{\value{firstpage}}
%
%
%
\newcommand{\Submitted}{July 15, 2011}
\newcommand{\Accepted}{October 3, 2011}
%
%
%
\newcommand{\SubjectClassification}{%
Primary 60J80; secondary 60K37, 60J60
}
\newcommand{\Keywords}{%
Branching diffusions in random environment, BDRE, supercriticality, survival probability
}
%
%
%
\newcommand{\Abstract}{%
  Supercritical branching processes in constant environment
  conditioned on eventual extinction are known to be
  subcritical branching processes.
  The case of random environment is more subtle.
  A supercritical branching diffusion in random environment (BDRE)
  conditioned on eventual extinction of the population
  is not a branching diffusion in a homogeneous environment.
  However the law of the population size of a supercritical BDRE
  (averaged over the environment)
  conditioned on eventual extinction is equal to the law of
  the population size of a subcritical BDRE
  (averaged over the environment).
  As a consequence, supercritical BDREs
  have a phase transition which is similar to a well-known
  phase transition of subcritical branching processes in random environment.
}
\input{ecp03.sty}
%

%
%
%
\section{Introduction and main results}%
\label{sec:Introduction}
Branching processes in random environment (BPREs) have attracted
considerable interest in recent years,
see e.g.~\cite{AfanasyevEtAl2005AOP,AfanasyevEtAl2011pre,DyakonovaEtAl2004}
and the references therein.
On the one hand this is due to the more realistic model
compared with classical branching processes.
On the other hand this is due to interesting properties
such as a phase transition in the subcritical regime.
Let us recall this phase transition.
In the strongly subcritical regime, the survival
probability of a BPRE $(Z_t^{(1)})_{t\geq0}$
scales like its expectation, that is,
$\P\big(Z_t^{(1)}>0\big)\sim const \cdot\E\big(Z_t^{(1)}\big)$
as $t\to\infty$ where $const$ is some constant in $(0,\infty)$.
In the weakly subcritical regime, the survival probability
decreases at a different exponential rate.
The intermediate subcritical regime is in between
the other two cases.
Understanding the differences of these three regimes is
one motivation of the literature cited above.
The main observation of this article is a
similar phase transition in the supercritical regime.

Let us introduce the model.
We consider a diffusion approximation of BPREs as this is
mathematically more convenient.
The diffusion approximation of BPREs
is due to Kurtz (1978)\nocite{Kurtz1978}
and had been conjectured (slightly inaccurately) by Keiding (1975)\nocite{Keiding1975}.
We follow
B\"o\-ing\-hoff and Hutzenthaler (2011)\nocite{BoeinghoffHutzenthaler2011pre}
and denote this diffusion approximation
as {\em branching diffusion in random environment} (BDRE).
For every $n\in\N:=\{1,2,\ldots\}$, let $(Z_k^{(n)})_{k\in\N_0}$ be a branching process
in the random environment $\big(Q_1^{(n)},Q_2^{(n)},\ldots\big)$
which is a sequence of independent, identically distributed offspring
distributions. If $m\big(Q_k^{(n)}\big)$ denotes the mean offspring number
for $k\in\N$, then $S_k^{(n)}:=\sqrt{n}\sum_{i=1}^{k-1}\log\big(m\big(Q_i^{(n)}\big)\big)$,
$k\in\N_0:=\{0,1,2,\ldots\}$ denotes the associated random walk where $n\in\N$.
Set $\floor{t}:=\max\{m\in\N_0\colon m\leq t\}$ for $t\geq0$.
Let the environment
be such that $\big(S_{\lfloor tn \rfloor}^{(n)}/\sqrt{n}\big)_{t\geq0}$
converges to a Brownian motion $(S_t)_{t\geq0}$ with infinitesimal drift $\al\in\R$
and infinitesimal standard deviation $\sigma_e\in[0,\infty)$ as $n\to\infty$.
Furthermore assume that the mean offspring variance converges to $\sigma_b^2\in[0,\infty)$,
that is,
\begin{equation}
  \lim_{n\to\infty}\E\left[\sum_{k=0}^\infty \left(k-m\left(Q_1^{(n)}\right)\right)^2 Q_1^{(n)}(k)\right]
  =\sigma_b^2.
\end{equation}
If $Z_0^{(n)}/n\to z\in[0,\infty)$ as $n\to\infty$ and if a third moment condition holds, then
  \begin{equation} \label{eq:diffusion.approximation}
    \roB{\frac{Z_{\lfloor tn \rfloor}^{(n)}}{n},
         \frac{S_{\lfloor tn \rfloor}^{(n)}}{\sqrt{n}}}_{t\geq0}
    \varwlimn\rob{Z_t,S_t}_{t\geq0}
  \end{equation}
  in the Skorohod topology (see e.g.~\cite{EthierKurtz1986}) where the limiting diffusion is
  the unique solution of the stochastic differential equations (SDEs)
\begin{equation}  \begin{split}  \label{eq:BDRE}
 dZ_t&=\frac 12 \sigma_e^2 Z_t dt+Z_t dS_t+\sqrt{\sigma_b^2 Z_t}dW^{(b)}_t\\
 dS_t&=\alpha dt +\sqrt{\sigma_e^2}dW^{(e)}_t
\end{split}     \end{equation}
for $t\geq0$ where $Z_0=z$ and $S_0=0$.
The processes $(W^{(b)}_{t})_{t\geq0}$
and $(W^{(e)}_{t})_{t\geq0}$ are independent standard Brownian motions.
Throughout the paper the notations $\P^{z}$ and $\E^{z}$ refer to $Z_0=z$ and $S_0=0$
for $z\in[0,\infty)$.
The diffusion approximation~\eqref{eq:diffusion.approximation}
is due to Kurtz (1978)\nocite{Kurtz1978} (see also~\cite{BoeinghoffHutzenthaler2011pre}).
Note that the random environment affects the limiting diffusion
only through
the mean branching variance $\sigma_b^2$ and through the associated random walk.

We denote the process $(S_t)_{t\geq0}$ as {\em associated Brownian motion}.
This process plays a central role.
For example it determines the conditional expectation of $Z_t$
\begin{equation}  \label{eq:expectation_BDRE}
  \E^{z}\sqb{Z_t|(S_s)_{s\leq t}}
  =z\exp\rob{S_t}
\end{equation}
for every $z\in[0,\infty)$ and $t\geq0$.
Moreover the infinitesimal drift $\al$ of the associated Brownian motion
determines the type of criticality.
The BDRE~\eqref{eq:BDRE} is supercritical (i.e. positive survival probability)
if $\al>0$,
critical if $\al=0$ and
subcritical if $\al<0$,
see Theorem 5 of B\"oinghoff and Hutzenthaler (2011)\nocite{BoeinghoffHutzenthaler2011pre}.
We will refer to $\al$ as criticality parameter.

Afanasyev (1979)\nocite{Afanasyev1979} was the first
to discover different regimes for the survival probability
of a BPRE
in the
subcritical regime
(see \cite{AfanasyevEtAl2005AOP,AfanasyevEtAl2005SPA, Vatutin2004,AfanasyevEtAl2011pre}
for recent articles).
The following characterisation for the BDRE~\eqref{eq:BDRE} is
due to B\"oinghoff and Hutzenthaler (2011)\nocite{BoeinghoffHutzenthaler2011pre}.
The survival probability of $(Z_t)_{t\geq0}$ decays like the expectation,
that is,
$\P(Z_t>0)\sim const\cdot\E(Z_t)=const\cdot \exp\big((\al+\tfrac{\sigma_e^2}{2})t\big)$ as $t\to\infty$,
if and only if $\al<-\sigma_e^2$
(strongly subcritical regime).
In the intermediate subcritical regime $\al=-\sigma_e^2$,
we have that $\P(Z_t>0)\sim const\cdot t^{-\frac{1}{2}} \exp\rob{-\frac{\sigma_e^2}{2}t}$
as $t\to\infty$.
Finally the survival probability decays like $\P(Z_t>0)\sim const\cdot
t^{-\frac{3}{2}}\exp\rob{-\frac{\al^2}{2\sigma_e^2}t}$ as $t\to\infty$
in the weakly subcritical regime $\al\in(-\sigma_e^2,0)$.

This article concentrates on the supercritical regime $\al>0$.
Our main observation is that there is a phase transition which
is similar to the subcritical regime.
Such a phase transition has not been reported
for BPREs yet.
We condition on the event $\{Z_\infty=0\}=\{\lim_{t\to\infty} Z_t=0\}$ of eventual extinction
and propose the following notation.
If
$\P(Z_t>0|Z_\infty=0)\sim const\cdot\E(Z_t|Z_\infty=0)$ as $t\to\infty$,
then we say that the BDRE $(Z_t,S_t)_{t\geq0}$ is
{\em strongly supercritical}.
If the probability of survival up to time $t\geq0$
conditioned on eventual extinction decays at a different exponential
rate as $t\to\infty$, then we refer to $(Z_t,S_t)_{t\geq0}$ as
{\em weakly supercritical}.
The intermediate regime is referred to as
{\em intermediate supercritical} regime.
Our first theorem provides the following characterisation.
The BDRE is
strongly supercritical if $\al>\sigma_e^2$,
intermediate supercritical if $\al=\sigma_e^2$
and
weakly supercritical if $\al\in(0,\sigma_e^2)$.
\begin{theorem}\label{thm:asymptotic_survival}
  Assume $\al,\sigma_e,\sigma_b\in(0,\infty)$.
  Let $(Z_t,S_t)_{t\geq0}$ be the unique solution of~\eqref{eq:BDRE}
  with $S_0=0$.
  Then
 \begin{align}
 \limt\sqrt{t}^3 e^{\frac{\al^2}{2\sigma_e^2}t}\,\P^z\Bigl(Z_t>0\,\Big|\,Z_\infty=0\Bigr)
   &=\frac{8}{\sigma_e^3}\int_0^\infty f(za)\phi_\beta(a)\,da>0
   &\text{if }\al\in(0,\sigma_e^2)\label{thweak}\\
 \limt\sqrt{t} e^{\frac{\sigma_e^2}{2}t}\,\P^z\Bigl(Z_t>0\,\Big|\,Z_\infty=0\Bigr)
   &=z\,\frac{\sqrt{2}\sigma_e}{\sqrt{\pi}\sigma_b^2}>0
   &\text{if }\al=\sigma_e^2\label{thinter}\\
 \limt e^{\rob{\al-\frac{\sigma_e^2}{2}}t}\,\P^z\Bigl(Z_t>0\,\Big|\,Z_\infty=0\Bigr)
   &=z\,2\frac{\al-\sigma_e^2}{\sigma_b^2}>0
   &\text{if }\al>\sigma_e^2\label{thstrong}
 \end{align}
 for every $z\in(0,\infty)$
 where $\beta:=\tfrac{2\al}{\sigma_e^2}$
 and where $\phi_\beta\colon(0,\infty)\to(0,\infty)$ is defined as
 \begin{equation} \label{defphibeta} \begin{split}
    \phi_\beta(a)=\int_0^\infty \int_0^\infty \frac{1}{\sqrt{2}\pi}\Gamma\Big(\frac{\beta+2}{2}\Big)
             e^{-a} a^{-\beta/2} u^{(\beta-1)/2}e^{-u}
      \frac{\sinh(\xi)\cosh(\xi) \xi}{(u+a(\cosh(\xi))^2)^{(\beta+2)/2}}d\xi\,du
 \end{split}    \end{equation}
 for every $a\in(0,\infty)$.
\end{theorem}
\noindent
The proof is deferred to Section~\ref{sec.proofs}.

Let us recall the behavior of Feller's branching diffusion, that is,
\eqref{eq:BDRE} with $\sigma_e=0$,
which is a branching diffusion in a constant environment.
The supercritical Feller diffusion conditioned on eventual extinction
agrees in distribution with a subcritical Feller diffusion.
This is a general property of branching processes in constant environment,
see Jagers and Lager{\aa}s (2008)\nocite{JagersLageras2008}
for the case of general branching processes
(Crump-Mode-Jagers processes).
Knowing this, Theorem~\ref{thm:asymptotic_survival} might not be surprising.
However, the case of random environment is different.
It turns out that the supercritical BDRE $(Z_t,S_t)_{t\geq0}$ conditioned on $\{Z_\infty=0\}$
is a two-dimensional diffusion which
does {\em not} satisfy the SDE~\eqref{eq:BDRE} and is not a branching diffusion
in homogeneous random environment 
if $\sigma_b>0$ and if $\sigma_e>0$.
More precisely, the associated Brownian motion $(S_t)_{t\geq0}$ conditioned on $\{Z_\infty=0\}$
has drift which depends on the current population size.

\begin{theorem}  \label{thm:sup.conditioned.E}
  Let $\sigma_e\in(0,\infty)$, let $\sigma_b,z\in[0,\infty)$
  and assume $\sigma_b+z>0$.
  If $\big(Z_t,S_t\big)_{t\geq0}$
  is the solution of~\eqref{eq:BDRE}
  with criticality parameter $\al\in(0,\infty)$,
  then
  \begin{equation}
    \Law{\left(Z_t,S_t\right)_{t\geq0}\big|Z_\infty=0}
    =
    \Law{\left(\Zv_t,\Sv_t\right)_{t\geq0}}
  \end{equation}
  where $\big(\Zv_t,\Sv_t\big)_{t\geq0}$ is a two-dimensional diffusion
  satisfying $\Zv_0=Z_0, \Sv_0=0$ and
  \begin{equation}  \begin{split}  \label{eq:Zvee}
    d\Zv_t&=\left(
         \frac 12 \sigma_e^2 -2\al\frac{\sigma_b^2}{\sigma_e^2 \Zv_t+\sigma_b^2}
            \right)\Zv_t dt
         +\Zv_t d\Sv_t+\sqrt{\sigma_b^2 \Zv_t}dW^{(b)}_t\\
    d\Sv_t&=\left(
                \alpha -2\al\frac{\sigma_e^2 \Zv_t}{\sigma_e^2 \Zv_t+\sigma_b^2}
            \right)dt +\sqrt{\sigma_e^2}dW^{(e)}_t
  \end{split}     \end{equation}
  for $t\geq0$.
\end{theorem}
\noindent
The proof is deferred to Section~\ref{sec.proofs}.

It is rather intuitive that the conditioned process is not a subcritical
BDRE if $\sigma_b>0$.
The supercritical BDRE has a positive probability of extinction.
Thus extinction does not require the associated Brownian motion to have negative drift.
As long as the BDRE stays small, extinction is possible despite the positive
drift of the associated Brownian motion.
Note that if $\Zv_t$ is small for some $t\geq0$, then the drift
term of $\Sv_t$ is close to $\al$.
Being doomed to extinction, the conditioned process
$(\Zv_t)_{t\geq0}$ is not allowed to grow to infinity.
If $\Zv_t$ is large for some $t\geq0$, then the drift term of
$\Sv_t$ is close to $-\al$ which leads a decrease of $(\Zv_t)_{t\geq0}$.
The situation is rather different in the case $\sigma_b=0$.
Then the extinction probability of the BDRE is zero.
So the drift of $(\Sv_t)_{t\geq0}$ needs to be negative in order to guarantee
$\Zv_t\to0$ as $t\to\infty$.
It turns out that if $\sigma_b=0$, then the drift of $(\Sv_t)_{t\geq0}$ is $-\al$
and $(\Zv_t,\Sv_t)_{t\geq0}$ is a subcritical BDRE with criticality parameter $-\al$.

We have seen that conditioning a supercritical BDRE on extinction
does -- in general -- not result in a subcritical BDRE.
However, if we condition $(Z_t,S_t)_{t\geq0}$ on
$\{S_\infty=-\infty\}$, then the conditioned process turns out to be
a subcritical BDRE with criticality parameter $-\al$.
\begin{theorem}  \label{thm:sup.conditioned.ES}
  Let $\sigma_e\in(0,\infty)$, let $\sigma_b,z\in[0,\infty)$
  and assume $\sigma_b+z>0$.
  Let $\big(Z_t^{(\al)},S_t^{(\al)}\big)_{t\geq0}$
  be the solution of~\eqref{eq:BDRE}
  with criticality parameter $\al\in\R$.
  If $\al>0$, then
  \begin{equation}
    \Law{\left(Z_t^{(\al)},S_t^{(\al)}\right)_{t\geq0}\big|S_\infty=-\infty}
    =
    \Law{\left(Z_t^{(-\al)},S_t^{(-\al)}\right)_{t\geq0}}
  \end{equation}
  where $Z_0^{(-\al)}=Z_0^{(\al)}$.
\end{theorem}
\noindent
The proof is deferred to Section~\ref{sec.proofs}.

Now we come to a somewhat surprising observation.
We will show that the law of $(Z_t)_{t\geq0}$ conditioned
on eventual extinction agrees in law with the law
of the population size of a subcritical BDRE.
More formally, inserting the second equation of~\eqref{eq:Zvee}
into the equation for $d\Zv_t$ we see that
\begin{equation}  \label{eq:Zvee.alone}
    d\Zv_t=\left(
         \frac 12 \sigma_e^2 -\al
            \right)\Zv_t dt
         +\sigma_e\Zv_t dW_t^{(e)} +\sqrt{\sigma_b^2 \Zv_t}dW^{(b)}_t
\end{equation}
for $t\geq0$.
This is the SDE for the population size of a subcritical BDRE with
criticality parameter $-\al$.
As the solution of \eqref{eq:Zvee.alone} is unique,
this proves the following corollary of
Theorem~\ref{thm:sup.conditioned.E}.

\begin{corollary}  \label{thm:sup.conditioned.E.alone}
  Let $\sigma_e\in(0,\infty)$, let $\sigma_b,z\in[0,\infty)$
  and assume $\sigma_b+z>0$.
  Let $\big(Z_t^{(\al)},S_t^{(\al)}\big)_{t\geq0}$
  be the solution of~\eqref{eq:BDRE}
  with criticality parameter $\al$ for every $\al\in\R$.
  If $\al>0$, then
  the law of the BDRE with criticality parameter $\al$ conditioned on extinction
  agrees with the law of the BDRE with criticality parameter $-\al$,
  that is,
  \begin{equation}
    \Law{(Z_t^{(\al)})_{t\geq0}\big|Z_\infty=0}
    =
    \Law{(Z_t^{(-\al)})_{t\geq0}}
  \end{equation}
  where $Z_0^{(-\al)}=Z_0^{(\al)}$.
\end{corollary}

So far we considered the event of extinction.
Next we condition the BDRE on the event 
$\{Z_\infty>0\}
:=
\{\limt\, Z_t=\infty\}$
of non-extinction.
Define $U\colon[0,\infty)\to[0,\infty)$ by
  \begin{equation}
    U(z):=\left(\sigma_e^2 z+\sigma_b^2\right)^{-\frac{2\al}{\sigma_e^2}}
  \end{equation}
for $z\in[0,\infty)$.
We agree on the convention that
\begin{equation}
  \frac{c}{0}:=\begin{cases}
                 \infty&\text{if }c\in(0,\infty]\\
                 0&\text{if }c=0
               \end{cases}
  \quad
  \frac{c}{\infty}:=0 \text{ for }c\in[0,\infty)
  \quad\text{and that }0\cdot\infty=0.
\end{equation}

\begin{theorem}  \label{thm:sup.conditioned.NE}
  Let $\sigma_e\in(0,\infty)$, let $\sigma_b,z\in[0,\infty)$
  and assume $\sigma_b+z>0$.
  Let $\big(Z_t,S_t\big)_{t\geq0}$
  be the solution of~\eqref{eq:BDRE}
  with criticality parameter $\al>0$.
  Then
  \begin{equation}  \label{eq:sup.conditioned.NE}
    \Law{\left(Z_t,S_t\right)_{t\geq0}\big|Z_\infty>0}
    =
    \Law{\left(\Zh_t,\Sh_t\right)_{t\geq0}}
  \end{equation}
  where $\big(\Zh_t,\Sh_t\big)_{t\geq0}$ is a two-dimensional diffusion
  satisfying $\Zh_0=Z_0, \Sh_0=0$ and
  \begin{equation}  \begin{split}  \label{eq:Zhat}
    d\Zh_t&=\left(
         \frac 12 \sigma_e^2 +2\al
            \frac{\sigma_b^2}{\sigma_e^2 \Zh_t+\sigma_b^2}
            \frac{U\big(\Zh_t\big)}{U(0)-U\big(\Zh_t\big)}
           \right)\Zh_t\, dt
         +\Zh_t d\Sh_t+\sqrt{\sigma_b^2 \Zh_t}dW^{(b)}_t\\
    d\Sh_t&=\left(
                \alpha +2\al\frac{\sigma_e^2 \Zh_t}{\sigma_e^2 \Zh_t+\sigma_b^2}
                \frac{U\big(\Zh_t\big)}{U(0)-U\big(\Zh_t\big)}
            \right)\,dt
            +\sqrt{\sigma_e^2}dW^{(e)}_t
  \end{split}     \end{equation}
  for $t\geq0$.
  The law of $(Z_t)_{t\geq0}$ conditioned on non-extinction
  satisfies that
  \begin{equation}
    \Law{(Z_t)_{t\geq0}\big|Z_\infty>0}
    =
    \Law{(\Zh_t)_{t\geq0}}
  \end{equation}
  where $\big(\Zh_t\big)_{t\geq0}$ is the solution of the one-dimensional SDE
  satisfying $\Zh_0=Z_0$ and
  \begin{equation}  \begin{split}  \label{eq:Zhat.quenched}
    d\Zh_t&=\left(
         \frac 12 \sigma_e^2 +\al+2\al
            \frac{U\big(\Zh_t\big)}{U(0)-U\big(\Zh_t\big)}
         \right)\Zh_t\, dt
         +\sigma_e\Zh_t dW_t^{(e)}+\sqrt{\sigma_b^2 \Zh_t}dW^{(b)}_t
  \end{split}     \end{equation}
  for $t\geq0$.
\end{theorem}
\noindent
The proof is deferred to Section~\ref{sec.proofs}.

On the event of non-extinction,
the population size $Z_t$ of a supercritical  BDRE grows like its
expectation $\E(Z_t|S_t)$  as $t\to\infty$.
\begin{theorem}  \label{thm:exponential.growth}
  Let $\sigma_e\in(0,\infty)$, let $\sigma_b,z\in[0,\infty)$
  and assume $\sigma_b+z>0$.
  Let $(Z_t,S_t)_{t\geq0}$ be the solution of~\eqref{eq:BDRE}
  with criticality parameter $\al\in\R$.
  Then $\left(Z_t/e^{S_t}\right)_{t\geq0}$ is a nonnegative martingale.
  Consequently for every initial value $Z_0=z\in[0,\infty)$
  there exists a random variable $Y\colon\Omega\to[0,\infty)$ such that
  \begin{equation} 
    \frac{Z_t}{e^{S_t}}\lra Y\qquad\text{as }t\to\infty
    \quad\text{almost surely}.
  \end{equation}
  The limiting variable is zero if and only if the BDRE goes to extinction,
  that is, $\P^z(Y=0)=\P^z(Z_\infty=0)$. In the supercritical case $\al>0$,
  the distribution of the limiting variable $Y$ satisfies that
  \begin{equation}  \label{eq:Laplace.Y}
    \E^z\Big[\exp\left(
       -\ld Y
     \right)\Big]
    =
    \E\bigg[\exp\bigg(
       -\frac{z}{\frac{\sigma_b^2}{\sigma_e^2}G_{\frac{2\al}{\sigma_e^2}}+\frac{1}{\ld}}
     \bigg)\bigg]
  \end{equation}
 for all $z,\ld\in[0,\infty)$
 where $G_\nu$ is gamma-distributed with
 shape parameter $\nu\in(0,\infty)$ and scale parameter $1$, that is,
 \begin{equation}  \label{eq:gamma.distribution}
   \P(G_\nu\in dx)= \frac{1}{\Gamma(\nu)}x^{\nu-1}e^{-x}\,dx
 \end{equation}
 for $x\in(0,\infty)$.
\end{theorem}
The proof is deferred to Section~\ref{sec.proofs}.
In particular, Theorem~\ref{thm:exponential.growth}
implies that $Z_\infty:=\limt Z_t$ exists almost surely
and that $Z_\infty\in\{0,\infty\}$ almost surely.
\section{Proofs}
\label{sec.proofs}
If $\sigma_b=0$ and $Z_0>0$, then the process
$(Z_t)_{t\geq0}$ does not hit $0$ in finite time
almost surely.
So the interval $(0,\infty)$ is a state space for
$(Z_t)_{t\geq0}$ if $\sigma_b=0$.
The following analysis works with the state space
$[0,\infty)$ for the case $\sigma_b>0$ and
with the state space
$(0,\infty)$ for the case $\sigma_b=0$.
To avoid case-by-case analysis
we assume $\sigma_b>0$ for the rest of this section.
One can check that our proofs also work
in the case $\sigma_b=0$ if the state space $[0,\infty)$
is replaced by $(0,\infty)$.

Inserting the associated Brownian motion $(S_t)_{t\geq0}$
into the diffusion equation of $(Z_t)_{t\geq0}$, we see that
$(Z_t)_{t\geq0}$ solves the SDE
\begin{equation}  \label{eq:Z}
 dZ_t=\Bigl(\alpha+\frac 12 \sigma_e^2\Bigr)Z_t\, dt
     +\sqrt{\sigma_e^2 Z_t^2}\,dW^{(e)}_t+\sqrt{\sigma_b^2 Z_t}\,dW^{(b)}_t
\end{equation}
for $t\in[0,\infty)$. 
One-dimensional diffusions are well-understood.
In particular the scale functions are known.
For the reason of completeness we derive a scale function
for~\eqref{eq:Z} in the following lemma.
The generator of $(Z_t,S_t)_{t\geq0}$ is the closure of the pregenerator
$\Gen\colon\C_0^2([0,\infty)\times\R)\to\C([0,\infty)\times\R)$
given by
\begin{equation}  \begin{split}
  \Gen f(z,s):=&\left(\al+\frac{\sigma_e^2}{2}\right)z\deldz f(z,s)
      +\al\delds f(z,s) 
      +\frac{1}{2}\left(\sigma_e^2 z^2+\sigma_b^2 z\right)\delzdzz f(z,s)
      \\
      &
      +\frac{1}{2}\sigma_e^2\delzdsz f(z,s)
      +\sigma_e^2z\delzdzds f(z,s)
\end{split}     \end{equation}
for all $z\in[0,\infty)$, $s\in\R$ and every $f\in\C_0^2\rob{[0,\infty)\times\R}$.

\begin{lemma}
  Assume $\sigma_e,\sigma_b,\al\in(0,\infty)$.
  Define the functions $U\colon[0,\infty)\to(0,\infty)$
  and $V\colon\R\to(0,\infty)$ through
  \begin{equation}
    U(z):=\left(\sigma_e^2 z+\sigma_b^2\right)^{-\frac{2\al}{\sigma_e^2}}
    \quad\text{ and }\quad
    V(s):=\exp\left(-\frac{2\al}{\sigma_e^2}s\right)
  \end{equation}
  for $z\in[0,\infty)$ and $s\in\R$.
  Then $U$ is a scale function for $(Z_t)_{t\geq0}$ and
  $V$ is a scale function for $(S_t)_{t\geq0}$,
  that is, $\Gen U\equiv 0$ and $\Gen V\equiv 0$
  so
  $\big(U(Z_t)\big)_{t\geq0}$ 
  and
  $\big(V(S_t)\big)_{t\geq0}$ 
  are martingales.
\end{lemma}
\begin{proof}
  Note that $U$ is twice continuously differentiable.
  Thus we get that
  \begin{equation}  \begin{split}
    \Gen \left(\sigma_e^2 z+\sigma_b^2\right)^{-\frac{2\al}{\sigma_e^2}}
    =
     &\left(\al+\frac{\sigma_e^2}{2}\right)z\cdot  
    \frac{-2\al}{\sigma_e^2}
     \left(\sigma_e^2 z+\sigma_b^2\right)^{-\frac{2\al}{\sigma_e^2}-1}
     \sigma_e^2
     \\
     &+\frac{1}{2}
     \left(\sigma_e^2 z^2+\sigma_b^2z\right)
    \frac{2\al}{\sigma_e^2}
    \left(\frac{2\al}{\sigma_e^2}+1\right)
     \left(\sigma_e^2 z+\sigma_b^2\right)^{-\frac{2\al}{\sigma_e^2}-2}
     \sigma_e^4
    \\
    =&
     \left(\sigma_e^2 z+\sigma_b^2\right)^{-\frac{2\al}{\sigma_e^2}-1}
     z
     \left(-2\al^2-\al\sigma_e^2+\frac{1}{2}4\al^2+\frac{1}{2}2\al\sigma_e^2\right)
     \\
    =&\, 0
  \end{split}     \end{equation}
  for all $z\in[0,\infty)$.
  Moreover $V$ is twice continuously differentiable and we obtain that
  \begin{equation}  \begin{split}
    \Gen \exp\left(-\frac{2\al}{\sigma_e^2}s\right)
    &=\al\exp\left(-\frac{2\al}{\sigma_e^2}s\right)\frac{-2\al}{\sigma_e^2}
      +\frac{1}{2}\sigma_e^2
      \exp\left(-\frac{2\al}{\sigma_e^2}s\right)
      \left(\frac{-2\al}{\sigma_e^2}\right)^2
      =0
  \end{split}     \end{equation}
  for all $s\in\R$.
  This shows $\Gen U\equiv0\equiv\Gen V$.
  Now  It\^o's formula implies that
  \begin{equation}  \begin{split}
    d U(Z_t) &=\Gen U(Z_t)\,dt+U^{'}(Z_t)
      \cdot
     \left(\sqrt{\sigma_e^2 Z_t^2}\,dW^{(e)}_t+\sqrt{\sigma_b^2 Z_t}\,dW^{(b)}_t
     \right)
    \\
    d V(S_t) &=\Gen V(S_t)\,dt+V^{'}(S_t)\sqrt{\sigma_e^2}d W_t^{(e)}
  \end{split}     \end{equation}
  for all $t\geq0$.
  This proves that
  $\big(U(Z_t)\big)_{t\geq0}$
  and
  $\big(V(S_t)\big)_{t\geq0}$
  are martingales.
\end{proof}
\begin{lemma} \label{l:semigroup}
  Assume $\sigma_e,\sigma_b,\al\in(0,\infty)$.
  Then the semigroup
  of the BDRE $(Z_t,S_t)_{t\geq0}$
  conditioned on extinction satisfies that
  \begin{align}
    \E^{(z,s)}\left[ f(Z_t,S_t)\big| Z_\infty=0 \right]
    &=\frac{\E^{(z,s)}\left[ U(Z_t) f(Z_t,S_t) \right]
          }{U(z)},
    \label{eq:semigroup.conditioned.E}
  \intertext{
  the semigroup
  of the BDRE $(Z_t,S_t)_{t\geq0}$
  conditioned on $\{S_\infty=-\infty\}$ satisfies that
  }
    \E^{(z,s)}\left[ f(Z_t,S_t)\big| S_\infty=-\infty \right]
    &=\frac{\E^{(z,s)}\left[ V(S_t) f(Z_t,S_t) \right]
          }{V(s)}
    \label{eq:semigroup.conditioned.ES}
  \intertext{
  and the semigroup
  of the BDRE $(Z_t,S_t)_{t\geq0}$
  conditioned on $\{Z_\infty>0\}$ satisfies that
  }
    \E^{(z,s)}\left[ f(Z_t,S_t)\big| Z_\infty>0 \right]
    &=\frac{\E^{(z,s)}\left[ \left(U(0)-U(Z_t)\right) f(Z_t,S_t) \right]
          }{U(0)-U(z)}
    \label{eq:semigroup.conditioned.NE}
  \end{align}
  for every $z\in[0,\infty)$, $s\in\R$, $t\geq0$ and every 
  bounded measurable function $f\colon[0,\infty)\times\R\to\R$.
\end{lemma}
\begin{proof}
  Define the first hitting time $T_x(\eta):=\inf\{t\geq0\colon \eta_t=x\}$ of
  $x\in\R$ for every continuous path $\eta\in\C\rob{[0,\infty),\R)}$.
  As $V$ is a scale function for $(S_t)_{t\geq0}$, the optional sampling
  theorem implies that
  \begin{equation}
    \P^s\left(
        T_{-N}(S)<\infty
        \right)
    =\limK
    \P^s\left(
        T_{-N}(S)<T_K(S)
        \right)
    =\limK
      \frac{V(K)-V(s)}{V(K)-V(-N)}
    =\frac{V(s)}{V(-N)}
  \end{equation}
  for all $s\in\R$  and $N\in\N$,
  see Section $6$ in \cite{KT2} for more details.
  Thus we get that
  \begin{equation}  \begin{split}
    \E^{(z,s)}\left[ f(Z_t,S_t)\big| S_\infty=-\infty \right]
    &=\limN
    \E^{(z,s)}\left[ f(Z_t,S_t)\big| T_{-N}(S)<\infty \right]
    \\
    &=\limN
    \frac{\E^{(z,s)}\left[ f(Z_t,S_t)\P^{S_t}\left(T_{-N}(S)<\infty\right) \right]}
         {\P^{(z,s)}\left(T_{-N}(S)<\infty\right) }
    \\
    &=\frac{\E^{(z,s)}\left[ f(Z_t,S_t) V(S_t) \right]
          }{V(s)}
  \end{split}     \end{equation}
  for all $z\in[0,\infty)$, $s\in\R$  and $t\geq0$.
  The proof of the assertions~\eqref{eq:semigroup.conditioned.E}
  and~\eqref{eq:semigroup.conditioned.NE}
  is analogous.
  Note for the proof of~\eqref{eq:semigroup.conditioned.NE}
  that
  \begin{equation}  \begin{split}
    \P^z\left(Z_\infty>0\right)
    =
    \P^z\left(\limt Z_t=\infty\right)
    =\limN
    \P^z\left(
        T_{N}(Z)<T_0(Z)
        \right)
    =
    \frac{U(0)-U(z)}{U(0)}
  \end{split}     \end{equation}
  for every $z\in[0,\infty)$.
\end{proof}
\begin{proof}[\bf Proof of Theorem~\ref{thm:sup.conditioned.E}]
  It suffices to identify the generator $\check{\Gen}$ of the conditioned
  process. This generator is the time derivative of the
  semigroup of the conditioned process at $t=0$.
  Let $f\in\C^2_0\rob{[0,\infty)\times\R,\R}$
  be fixed.
  Define
  $f_z(z,s):=\deldz f(z,s)$,
  $f_s(z,s):=\delds f(z,s)$,
  $f_{zz}(z,s):=\delzdzz f(z,s)$,
  $f_{ss}(z,s):=\delzdsz f(z,s)$
  and
  $f_{zs}(z,s):=\delzdzds f(z,s)$
  for $z\in[0,\infty)$ and $s\in\R$.
  Lemma~\ref{l:semigroup} implies that
  \begin{equation}  \begin{split}  \label{eq:conditioned.calc.generator}
  \lefteqn{
   \check{\Gen}f(z,s)
   }\\
   &=\lim_{h\to0}\frac{\E^{(z,s)}\left[ U(Z_h) f(Z_h,S_h)-U(z)f(z,s) \right]/U(z)}{h}
   =\frac{\Gen(U\cdot f)(z,s)}{U(z)}
   \\
   &=\frac{1}{U(z)}\Big[\Big(\al+\frac{\sigma_e^2}{2}\Big)z
         \left(U^{'}f+Uf_z\right)(z,s)
         +\al \left(Uf_{s}\right)(z,s)
        +\frac{\sigma_e^2}{2}\left(U f_{ss}\right)(z,s)
   \\
   &\quad\quad+\frac{1}{2}\left(\sigma_e^2 z^2+\sigma_b^2z\right)
        \left(U^{''}f+2U^{'}f_z+Uf_{zz}\right)(z,s)
        +\sigma_e^2 z\left(U^{'}f_{s}+U f_{zs}\right)(z,s)
          \Big]
   \\
   &=\frac{1}{U(z)}\left(\Gen f(z,s)\right)U(z)
    +\frac{1}{U(z)}\left(\Gen U(z)\right)f(z,s)
   \\
   &\qquad
    + \left(\sigma_e^2 z^2+\sigma_b^2z\right)
        \left(\frac{U^{'}}{U}f_z\right)(z,s)
        +\sigma_e^2 z\left(\frac{U^{'}}{U}f_{s}\right)(z,s)
  \end{split}     \end{equation}
  for all $z\in[0,\infty)$ and $s\in\R$.
  Now we exploit that $\Gen U\equiv0$ and that
  $\tfrac{U^{'}}{U}(z)=-2\al/(\sigma_e^2 z +\sigma_b^2)$ for $z\in[0,\infty)$
  to obtain that
  \begin{equation}  \begin{split}
   \check{\Gen}f(z,s)
   &=\Gen f(z,s)-2\al z f_z(z,s)
      -2\al\frac{\sigma_e^2 z}{\sigma_e^2 z+\sigma_b^2}f_s(z,s)
   \\
   &=\left(-\al+\frac{\sigma_e^2}{2}\right)z f_z(z,s)
      +\frac{1}{2}\left(\sigma_e^2 z^2+\sigma_b^2 z\right)f_{zz}(z,s)
      \\
      &\qquad
      +\left(\al-2\al\frac{\sigma_e^2 z}{\sigma_e^2 z+\sigma_b^2}\right)
       f_s(z,s) 
      +\frac{1}{2}\sigma_e^2 f_{ss}(z,s)
      +\sigma_e^2z f_{zs}(z,s)
   \\
  \end{split}     \end{equation}
  for all $z\in[0,\infty)$, $s\in\R$
  and
  all $f\in\C^2_0\rob{[0,\infty)\times\R,\R}$.
  This is the generator of the process~\eqref{eq:Zvee}.
  Therefore the BDRE conditioned on extinction has the same distribution
  as the solution of~\eqref{eq:Zvee}.
\end{proof}
\begin{proof}[\bf Proof of Theorem~\ref{thm:sup.conditioned.ES}]
  As in the proof of Theorem~\ref{thm:sup.conditioned.E} we identify
  the generator $\bar{\Gen}$ of the BDRE conditioned on $\{S_\infty=-\infty\}$.
  Similar arguments as in~\eqref{eq:conditioned.calc.generator}
  and $\Gen V\equiv 0$ result in
  \begin{equation*}  \begin{split}  \label{eq:conditioned.calc.generator.ES}
  \lefteqn{
   \bar{\Gen}f(z,s)
   }\\
   &=\Gen f(z,s)
    + \frac{\sigma_e^2}{2}
        2\left(\frac{V^{'}}{V}f_s\right)(z,s)
        +\sigma_e^2 z\left(\frac{V^{'}}{V}f_{z}\right)(z,s)
   \\
   &=\Gen f(z,s)
        -2\al f_s(z,s)
        -2\al z f_z(z,s)
   \\
   &=\left(-\al+\frac{\sigma_e^2}{2}\right)zf_z(z,s)
      -\al f_s(z,s) 
      +\frac{\sigma_e^2 z^2+\sigma_b^2 z}{2}f_{zz}(z,s)
      +\frac{\sigma_e^2}{2}f_{ss}(z,s)
      +\sigma_e^2zf_{zs}(z,s)
  \end{split}     \end{equation*}
  for all $z\in[0,\infty)$, $s\in\R$
  and
  all $f\in\C^2_0\rob{[0,\infty)\times\R,\R}$.
  This is the generator of the BDRE with criticality parameter $-\al$.
\end{proof}
\begin{proof}[\bf Proof of Theorem~\ref{thm:asymptotic_survival}]
  The assertion follows from Corollary~\ref{thm:sup.conditioned.E.alone}
  and from Theorem 5 of B\"o\-ing\-hoff and Hutzenthaler
  (2011)\nocite{BoeinghoffHutzenthaler2011pre}.
\end{proof}
\begin{proof}[\bf Proof of Theorem~\ref{thm:exponential.growth}]
  It\^{o}'s formula implies that
  \begin{equation}  \begin{split}
    d\frac{Z_t}{e^{S_t}}
    &=e^{-S_t}dZ_t-e^{-S_t}Z_tdS_t+\frac{1}{2}e^{-S_t}Z_t\sigma_e^2\,dt
     -e^{-S_t}Z_t\sigma_e^2\,dt
    \\
    &= e^{-S_t}\frac{\sigma_e^2}{2}Z_t\,dt
      +e^{-S_t}\sqrt{\sigma_b^2 Z_t}dW_t^{(b)}
      +\frac{1}{2}e^{-S_t}Z_t\sigma_e^2\,dt
      -e^{-S_t}Z_t\sigma_e^2\,dt
    \\
    &= e^{-S_t}\sqrt{\sigma_b^2 Z_t}dW_t^{(b)}
  \end{split}     \end{equation}
  for all $t\geq0$.
  Therefore $\left(Z_t/\exp(S_t)\right)_{t\geq0}$ is a nonnegative martingale.
  The martingale convergence theorem implies the existence
  of a random variable $Y\colon\Omega\to[0,\infty)$ such that
  \begin{equation} 
    \frac{Z_t}{e^{S_t}}\lra Y\qquad\text{as }t\to\infty
    \quad\text{almost surely.}
  \end{equation}
  If $\al\leq0$, then $Z_\infty=0$ almost surely, which implies
  $Y=0$ almost surely.

  It remains to determine the distribution of $Y$ in the supercritical regime $\al>0$.
  Fix $z\in[0,\infty)$ and $\ld\in[0,\infty)$.
  Dufresne (1990)\nocite{Dufresne1990} (see also~\cite{Yor1992JAP}) showed that
  \begin{equation}
    \int_0^\infty \exp\left(-\al s-\sigma_e W_{s}^{(e)}\right)\,ds
    \overset{\rm d}{=} \frac{2}{\sigma_e^2}G_{\frac{2\al}{\sigma_e^2}}.
  \end{equation}
  Moreover we exploit an explicit formula for the Laplace transform
  of the BDRE~\eqref{eq:BDRE} conditioned on the environment, see
  Corollary 3 of B\"oinghoff and Hutzenthaler (2011)\nocite{BoeinghoffHutzenthaler2011pre}.
  Thus we get that
  \begin{equation}  \begin{split}
    \E^z\left[\exp\left(
      -\ld Y
    \right)\right]
    &=
    \limt
    \E^z\left[\exp\left(
      -\ld\frac{Z_t}{e^{S_t}}
    \right)\right]
    =
    \limt
    \E^z\left[
      \E^z\left[\exp\left(
      -\ld\frac{Z_t}{e^{S_t}}
    \right)\Big|\left(S_s\right)_{s\in[0,t]}\right]
    \right]
    \\
    &=
    \limt
    \E\left[
    \exp\bigg(-\frac{z}{\int_0^t \frac{\sigma_b^2}{2}
         \exp\big(-S_s \big) ds
         + \frac{\exp(S_t)}{\lambda} \exp(-S_t)
         }
         \bigg)
    \right]
    \\
    &=
    \E\bigg[
    \exp\bigg(-\frac{z}{ \frac{\sigma_b^2}{2}\int_0^\infty
         \exp\big(-\al s-\sigma_e W_s^{(e)} \big) ds
         + \frac{1}{\lambda} 
         }
         \bigg)
    \bigg]
    \\
    &=
    \E\bigg[
    \exp\bigg(-\frac{z}{ \frac{\sigma_b^2}{\sigma_e^2}G_{2\al/\sigma_e^2}
         + \frac{1}{\lambda} 
         }
         \bigg)
    \bigg].
  \end{split}     \end{equation}
  This shows~\eqref{eq:Laplace.Y}.
  Letting $\ld\to\infty$ we conclude that
  \begin{equation}
    \P^z\left(Y=0\right)=
    \E\bigg[
    \exp\bigg(-\frac{z}{ \frac{\sigma_b^2}{\sigma_e^2}G_{2\al/\sigma_e^2}
         }
         \bigg)
       \bigg]
    =\P^z\left(Z_\infty=0\right).
  \end{equation}
  The last equality follows from Theorem 5 of \cite{BoeinghoffHutzenthaler2011pre}.
\end{proof}
\begin{proof}[\bf Proof of Theorem~\ref{thm:sup.conditioned.NE}]
  Analogous to the proof of Theorem~\ref{thm:sup.conditioned.E}, we identify
  the generator $\hat{\Gen}$ of the BDRE conditioned on $\{Z_\infty>0\}$.
  Note that
  \begin{equation}
    \frac{-U^{'}(z)}{U(0)-U(z)}
    =\frac{2\al}{\sigma_e^2z+\sigma_b^2}
     \frac{U(z)}{U(0)-U(z)}
  \end{equation}
  for all $z\in[0,\infty)$.
  Similar arguments as in~\eqref{eq:conditioned.calc.generator}
  and $\Gen U\equiv 0$ result in
  \begin{equation*}  \begin{split}  \label{eq:conditioned.calc.generator.NE}
  \lefteqn{
   \hat{\Gen}f(z,s)
   }\\
   &=\Gen f(z,s)
    + \left(\sigma_e^2 z^2+\sigma_b^2z\right)
         \frac{-U^{'}(z)}{U(0)-U(z)}f_z(z,s)
        +\sigma_e^2 z \frac{-U^{'}(z)}{U(0)-U(z)}f_{s}(z,s)
   \\
   &=\left(\al+2\al
      \frac{U(z)}{U(0)-U(z)}
      + \frac{\sigma_e^2}{2}
      \right)zf_z(z,s)
      +\left(\al+\sigma_e^2 z\frac{2\al}{\sigma_e^2 z+\sigma_b^2}
       \frac{U(z)}{U(0)-U(z)}
       \right)
       f_s(z,s) 
   \\
   &\qquad
      +\frac{\sigma_e^2 z^2+\sigma_b^2 z}{2}f_{zz}(z,s)
      +\frac{\sigma_e^2}{2}f_{ss}(z,s)
      +\sigma_e^2zf_{zs}(z,s)
  \end{split}     \end{equation*}
  for all $z\in[0,\infty)$, $s\in\R$
  and
  all $f\in\C^2_0\rob{[0,\infty)\times\R,\R}$.
  Comparing with~\eqref{eq:Zhat}, we see that $\hat{\Gen}$ 
  is the generator of~\eqref{eq:Zhat} which
  implies~\eqref{eq:sup.conditioned.NE}.
  Inserting $d\hat{S}_t$ into the equation of $d\hat{Z}_t$
  for $t\in[0,\infty)$ shows that
  $(\hat{Z}_t)_{t\geq0}$ solves the SDE~\eqref{eq:Zhat.quenched}.
\end{proof}
\appendix
\section*{Acknowledgement}
We thank two anonymous referees for very helpful comments and suggestions.

\begin{thebibliography}{10}

\bibitem{Afanasyev1979}
{\sc Afanasyev, V.~I.}
\newblock On the survival probability of a subcritical branching process in a
  random environment.
\newblock {\em \upshape Dep. VINITI\/} (1979), No. M1794--79 (in Russian).

\bibitem{AfanasyevEtAl2011pre}
{\sc Afanasyev, V.~I., B\"oinghoff, C., Kersting, G., and Vatutin, V.~A.}
\newblock Limit theorems for a weakly subcritical branching process in a random
  environment.
\newblock {\em to appear in J. Theoret. Probab., DOI:
  10.1007/s10959-010-0331-6\/} (2010).

\bibitem{AfanasyevEtAl2005AOP}
{\sc Afanasyev, V.~I., Geiger, J., Kersting, G., and Vatutin, V.~A.}
\newblock Criticality for branching processes in random environment.
\newblock {\em Ann. Probab. 33}, 2 (2005), 645--673.

\bibitem{AfanasyevEtAl2005SPA}
{\sc Afanasyev, V.~I., Geiger, J., Kersting, G., and Vatutin, V.~A.}
\newblock Functional limit theorems for strongly subcritical branching
  processes in random environment.
\newblock {\em Stochastic Process. Appl. 115}, 10 (2005), 1658--1676.

\bibitem{BoeinghoffHutzenthaler2011pre}
{\sc B\"oinghoff, C., and Hutzenthaler, M.}
\newblock Branching diffusions in random environment.
\newblock {\em http://arxiv.org/abs/1107.2773v1\/} (2011).

\bibitem{Dufresne1990}
{\sc Dufresne, D.}
\newblock The distribution of a perpetuity, with applications to risk theory
  and pension funding.
\newblock {\em Scand. Acturial. J. 1990}, 1 (1990), 39--79.

\bibitem{DyakonovaEtAl2004}
{\sc Dyakonova, E.~E., Geiger, J., and Vatutin, V.~A.}
\newblock On the survival probability and a functional limit theorem for
  branching processes in random environment.
\newblock {\em Markov Process. Related Fields 10}, 2 (2004), 289--306.

\bibitem{EthierKurtz1986}
{\sc Ethier, S.~N., and Kurtz, T.~G.}
\newblock {\em Markov processes: {C}haracterization and convergence}.
\newblock Wiley Series in Probability and Mathematical Statistics: Probability
  and Mathematical Statistics. John Wiley \& Sons Inc., New York, 1986.

\bibitem{JagersLageras2008}
{\sc Jagers, P., and Lager{\aa}s, A.~N.}
\newblock General branching processes conditioned on extinction are still
  branching processes.
\newblock {\em Electron. Commun. Probab. 13\/} (2008), 540--547.

\bibitem{KT2}
{\sc Karlin, S., and Taylor, H.~M.}
\newblock {\em A second course in stochastic processes}.
\newblock Academic Press Inc. [Harcourt Brace Jovanovich Publishers], New York,
  1981.

\bibitem{Keiding1975}
{\sc Keiding, N.}
\newblock Extinction and exponential growth in random environments.
\newblock {\em Theor. Population Biology 8\/} (1975), 49--63.

\bibitem{Kurtz1978}
{\sc Kurtz, T.~G.}
\newblock Diffusion approximations for branching processes.
\newblock In {\em Branching processes ({C}onf., {S}aint {H}ippolyte, {Q}ue.,
  1976)}, vol.~5 of {\em Adv. Probab. Related Topics}. Dekker, New York, 1978,
  pp.~269--292.

\bibitem{Vatutin2004}
{\sc Vatutin, V.~A.}
\newblock A limit theorem for an intermediate subcritical branching process in
  a random environment.
\newblock {\em Theory Probab. Appl. 48}, 3 (2004), 481--492.

\bibitem{Yor1992JAP}
{\sc Yor, M.}
\newblock Sur certaines fonctionnelles exponentielles du mouvement brownien
  r\'eel.
\newblock {\em J. Appl. Probab. 29}, 1 (1992), 202--208.

\end{thebibliography}
\def\cprime{$'$} \hyphenation{Sprin-ger}

\end{document}